\documentclass{amsart}
\usepackage[margin=1in]{geometry}%
\usepackage{xcolor}
\usepackage{setspace,fullpage}%
\usepackage{amssymb}
\usepackage{amsmath,amsthm,amscd,amssymb, mathrsfs}
\usepackage{latexsym}
\usepackage{color}

\numberwithin{equation}{section}

\theoremstyle{plain}
\newtheorem{theorem}{Theorem}[section]
\newtheorem{lemma}[theorem]{Lemma}
\newtheorem{corollary}[theorem]{Corollary}
\newtheorem{proposition}[theorem]{Proposition}

\theoremstyle{definition}
\newtheorem{definition}[theorem]{Definition}

\newtheorem{problem}[theorem]{Problem}
\newtheorem{example}[theorem]{Example}

\theoremstyle{remark}

\newtheorem{case[theorem]}{Case}

\newcommand{\nothing}[1]{{}}

\title[The quotient set of the distance set]{The quotient set of the quadratic distance set over finite fields}
\author{Alex Iosevich, Doowon Koh, and Firdavs Rakhmonov}

\address{Department of Mathematics\\
University of Rochester\\
Rochester, NY  14627 USA}
\email{iosevich@math.rochester.edu}

\address{Department of Mathematics\\
Chungbuk National University \\
Cheongju, Chungbuk 28644 Korea}
\email{koh131@chungbuk.ac.kr}

\address{Department of Mathematics\\
University of Rochester\\
Rochester, NY  14627 USA}
\email{frakhmon@ur.rochester.edu}

\thanks{
The research of the first listed author was supported in part by the National Science Foundation under grant no. HDR TRIPODS -1934962 and by the NSF DMS-2154232. The second listed author was supported by the National Research Foundation of Korea (NRF) grant funded by the Korea government (MSIT) (NO. RS-2023-00249597).} 


\begin{document}

\begin{abstract}
Let $\mathbb F_q^d$ be the $d$-dimensional vector space over the finite field $\mathbb F_q$ with $q$ elements. For each non-zero $r$ in $\mathbb F_q$ and $E\subset \mathbb F_q^d$, we  define $W(r)$ as the number of quadruples $(x,y,z,w)\in E^4$ such that $
Q(x-y)/Q(z-w)=r,$ where $Q$ is a non-degenerate quadratic form in $d$ variables over $\mathbb F_q.$
When $Q(\alpha)=\sum_{i=1}^d \alpha_i^2$ with $\alpha=(\alpha_1, \ldots, \alpha_d)\in \mathbb F_q^d,$ 
Pham (2022) recently used the machinery of group actions and  proved that if $E\subset \mathbb F_q^2$ with $q\equiv 3 \pmod{4}$ and $|E|\ge C q$, then we have $W(r)\ge c |E|^4/q$ for any non-zero square number  $r \in \mathbb F_q,$ where $C$ is a sufficiently large constant, $ c$ is some number between $0$ and $1,$ and $|E|$ denotes the cardinality of the set $E.$

In this article, we improve and extend Pham's result in two dimensions to arbitrary dimensions with general non-degenerate  quadratic distances. As a corollary of our results, we also generalize the sharp results on the Falconer type problem for the quotient set of distance set due to the first two authors and Parshall (2019). Furthermore, we provide improved constants for the size conditions of the underlying sets.

The key new ingredient is to relate the estimate of the $W(r)$ to a quadratic homogeneous variety in $2d$-dimensional vector space.
This approach is fruitful because it allows us to take advantage of Gauss sums which are more handleable than the Kloosterman sums appearing in the standard distance type problems. 
\end{abstract}
\maketitle

\section{Introduction} 

Let $\mathbb F_q^d, d\ge 2,$ be the $d$-dimensional vector space over the finite field $\mathbb F_q$ with $q$ elements. Throughout this paper, we assume that $q$ is a power of odd prime $p.$ 
Given a set $E$ in $\mathbb F_q^d$,  the distance set $\Delta(E)$ is defined by
$$ \Delta(E):=\{||x-y||\in \mathbb F_q: x,y \in E\},$$
where $||\alpha||=\sum_{i=1}^d \alpha_i^2$ for $\alpha=(\alpha_1, \ldots, \alpha_d)\in \mathbb F_q^d.$ 

In the finite field setting, the Falconer distance problem asks for the minimal exponent $\alpha>0$ such that  for any subset $E$ of $\mathbb F_q^d$ with $|E|\ge C q^\alpha,$  we have
$ |\Delta(E)|\ge c q.$ Here, and throughout this paper, $C>1$ denotes a sufficiently large constant, and  $ 0<c \le 1$ denotes some constant independent of $q$ and $|E|,$ where $|E|$ denotes the cardinality of $E.$  This problem was initially proposed by the first listed author and Rudnev \cite{IR07} as a finite field analogue of the Falconer distance problem in the Euclidean space. We notice that the formulation of the finite field Falconer problem was also motivated on the Erd\H{o}s distinct distances problem over finite fields due to  Bourgain, Katz and Tao \cite{BKT04}. Hence, the problem is also called the Erd\H{o}s-Falconer distance problem.  
We refer readers to \cite{Er,GK, SV05, GIS, Fa85, GIOW, DIOWZ, DZ, DGOWWZ} for precise definition, background knowledge, and recent progress on the Erd\H{o}s distinct distances problem and  the Falconer distance problem in the Euclidean setting.  

As a strong version of the Erd\H{o}s-Falconer distance problem, one can consider the Mattila-Sj\"olin distance problem over finite fields which is to determine the smallest threshold $\beta >0$ such that any subset $E$ of $\mathbb F_q^d$ with $|E|\ge Cq^\beta$ generates all possible distances, namely, $ \Delta(E)=\mathbb F_q.$ 
Using Fourier analytic machinery and the Kloosterman sum estimate, the first listed author and Rudnev  obtained the threshold $(d+1)/2$ for all dimensions $d\ge 2.$

\begin{theorem} [Iosevich and Rudnev, \cite{IR07}]\label{IRthm} \label{IRThm} Suppose that $E\subset \mathbb F_q^d, d\ge 2,$ 
and  $|E|> 2 q^{(d+1)/2}.$  Then we have  $\Delta(E)=\mathbb F_q.$
\end{theorem}

The threshold $(d+1)/2$ in Theorem \ref{IRthm} is the best currently known result on the Mattila-Sj\"olin distance problem over finite fields for all dimensions $d\ge 2.$ It is considered as an extremely hard problem to improve the $(d+1)/2$ result.  Moreover,  in the general setting of odd dimensions,  it gives the optimal threshold, which was proven in \cite{HIKR10}.
 
However, in even dimensions,  it has been believed that the exponent $(d+1)/2$ can be improved but a reasonable evidence or conjecture has not been stated in literature. In two dimensions, Murphy and Petridis \cite{MP19} showed that the threshold  cannot be lower than $4/3$ for the Mattila-Sj\"olin distance problem over finite fields.

On the other hand,  the first listed author and Rudnev \cite{IR07} conjectured that the threshold $(d+1)/2$ can be lower to $d/2$ for the Erd\H{o}s-Falconer distance problem in even dimensions, and  in two dimensions the threshold $4/3$ was proven by the authors in \cite{CEHIK09} (see also \cite{BHIPR}). Furthermore, when $q$ is prime, the exponent $4/3$ was improved to $5/4$ by Murphy,  Petridis,  Pham,  Rudnev, and  Stevenson \cite{MPPRS}. We also notice that the threshold $(d+1)/2$ cannot be improved for the Erd\H{o}s-Falconer distance problem in general odd dimensions (see also \cite{HIKR10}). 

The distance problems over finite fields have been extended  in various directions.
For example, Shparlinski \cite{Sh06} studied the size of the distance set  between two sets in $\mathbb F_q^d$ (see also \cite{KS12, KS15}). The generalized distance problems with polynomial distances were investigated by  the second listed author and Shen \cite{KS}  and Vinh  \cite{Vi13}.  Covert, the second listed author, and Pi \cite{CKP} studied the size of the $k$-distance set.
Hart and the first listed author \cite{HI08} extended the distance problem to the $k$-simplices problem over finite fields (see also \cite{BCCHIP, BIP14, Pa17, PPV, Vi12}). Shparlinski \cite{Sh16} addressed the Erd\H{o}s-Falconer distance problem for the sum of the distance set, namely $\Delta(E) + \Delta(E)$ (see also \cite{CKP21, KPSV21}).

Although lots of variants of the distance problems were extensively studied, the threshold $d/2$ for the set $E$ in $\mathbb F_q^d$ had not been addressed for any distance type problems until the first two authors and Parshall \cite{IKP} studied the following Mattila-Sj\H{o}lin problem for the quotient set of the distance set over finite fields: 
\begin{problem} [The Mattila-Sj\H{o}lin problem for the quotient set of the distance set]\label{prIKP} Given a set $E$ in $\mathbb F_q^d, d\ge 2,$ the quotient set of the distance set, denoted by $\frac{\Delta(E)}{\Delta(E)},$ is defined by
$$\frac{\Delta(E)}{\Delta(E)}:=\left\{ \frac{||x-y||}{||z-w||}: x,y, z, w\in E, ~ ||z-w||\ne 0\right\}.$$
Determine the smallest exponent $\gamma>0$ such that,  for any set $E \subset \mathbb F_q^d$ with $|E|\ge C q^\gamma$, we have
\begin{equation}\label{Qeq} \frac{\Delta(E)}{\Delta(E)}=\mathbb F_q.\end{equation}
\end{problem}

The aforementioned authors obtained the threshold $d/2$ in even dimensions on the Mattila-Sj\H{o}lin problem for the quotient set of the distance set over finite fields. More precisely, they proved the following result.

\begin{theorem} [Theorems 1.1, 1.2, \cite{IKP}]\label{IKPmthm} Let $E \subset \mathbb F_q^d, d\ge 2.$ Then the following statements hold:
\begin{enumerate} \item  [(i)]If $d\ge 2$ is even and $ |E|\ge 9 q^{d/2},$ then  $\dfrac{\Delta(E)}{\Delta(E)}=\mathbb F_q.$

\item [(ii)] If $d\ge 3$ is odd and $|E|\ge 6 q^{d/2},$ then  $\dfrac{\Delta(E)}{\Delta(E)}\supseteq(\mathbb F_q)^2,$
where $(\mathbb F_q)^2 := \{a^2: a\in \mathbb F_q\}.$  
\end{enumerate}
\end{theorem}


We shall write $\mathbb F_q^*$ for the set of all non-zero elements in $\mathbb F_q.$                                                  

As the main idea to deduce  Theorem \ref{IKPmthm}, the authors \cite{IKP} first observed that for any $r\in \mathbb F_q^*,$ we have 
$$r\in \frac{\Delta(E)}{\Delta(E)} \quad\mbox{if}~~  \sum_{t\in \mathbb F_q} v(t) v(rt) >v^2(0),$$ 
where
$v(t)$ is the number of the pairs $(x,y)\in E\times E$ such that $||x-y||=t,$ namely,
\begin{equation} \label{countingF} v(t):=\sum_{\substack{x, y\in E: \\ ||x-y||=t}} 1. \end{equation} 
Next, using the discrete Fourier analysis, they  estimated  a lower bound of $\sum_{t\in \mathbb F_q} v(t) v(rt)$ and an upper bound of $v^2(0).$
Finally,  the required size condition on the sets $E$ was obtained by comparing them.
Although the method of the proof led to the optimal threshold result on the Mattila-Sj\"olin  problem for the quotient set of the distance set,   it has two drawbacks below, as mentioned by Pham \cite{Ph}.

\smallskip

\begin{itemize}
\item The proof is too sophisticated and it requires large amount of calculation.

\smallskip

\item It is not clear from the proof that  how many quadruples $(x,y,z,w)\in E^4$ contribute to producing each element $r\in \mathbb F_q^*$ such that $\frac{||x-y||}{||z-w||}=r,$ namely, $r\in \frac{\Delta(E)}{\Delta(E)}.$
\end{itemize}

As a way to overcome the above issues, Pham \cite{Ph} utilized the machinery of group actions in two dimensions and obtained   a lower bound of $V(r)$ for any square number $r$ in $\mathbb F_q^*,$ 
where $V(r)$ denotes the number of the quadruples $(x,y,z,w)\in E^4$ such that $\frac{||x-y||}{||z-w||}=r,$ namely, 
\begin{equation}\label{DefV} V(r):=\sum_{\substack{x,y,z,w\in E:\\ \frac{||x-y||}{||z-w||}=r}} 1.\end{equation} 

As a consequence, he provided a short proof to deduce the following lower bound of $V(r)$ in two dimensions.

\begin{theorem} [Theorem 1.2, \cite{Ph}]\label{Phthm} Let $E$ be a subset of $\mathbb F_q^2.$   
Suppose that $|E|\ge C q$ with $q\equiv 3 \pmod{4}.$ Then, for any non-zero square number $r$ in $\mathbb F_q^*,$  we have
$$ V(r)\ge  \frac{c|E|^4}{q}.$$
In particular,  we have $\dfrac{\Delta(E)}{\Delta(E)} \supseteq (\mathbb F_q)^2:=\{a^2: a\in \mathbb F_q\}.$
\end{theorem}

Pham's approach, based on the group action, is  powerful in the sense that it gives a simple proof and an information about a lower bound of $V(r).$ However, his result, Theorem \ref{Phthm}, is limited to two dimensions with $-1$ non-square in $\mathbb F_q,$ and it gives us no information about $V(r)$ for a non-square $r$ in $\mathbb F_q^*.$ 

\smallskip

One of the main contributions of this paper is to improve and  extend Pham's result to all dimensions for arbitrary finite fields. In particular, we will work on the problem with the non-degenerate quadratic distances which generalize the usual distance.

The other contribution is to provide significantly improved explicit constants for size conditions on the underlying sets.
These are mainly based on our approach to view the problem for the set $E$ in $\mathbb F_q^d$ as that for the product set $E\times E$ associated with certain homogeneous variety of degree two in $\mathbb F_q^{2d}$ (see Section \ref{ProofMthm}). 
As a direct consequence of this method,  we will also address  a generalized version of   Theorem \ref{IKPmthm}, with more accurate constants.

To precisely state our main results, let us set up  some definitions and notations related to the quadratic distance.
\begin{definition} [Non-degenerate quadratic form]We identify a vector $x=(x_1, \ldots, x_d)\in \mathbb F_q^d$ with a $d\times 1$ matrix over $\mathbb F_q.$ 
A non-degenerate quadratic form $Q(x)$ is a homogeneous polynomial of degree 2 in $\mathbb F_q[x_1,\ldots, x_d]$, and 
is written by $Q(x)=x^{\top} A~ x =\sum\limits_{i,j=1}^d a_{ij} x_ix_j$ for some $d\times d$ symmetric matrix $A=[a_{ij}]$ with $\det(A)\ne 0,$
where $x^\top$ denotes the transpose of the column vector $x$ and  each entry $a_{ij}$ of $A$ belongs to $\mathbb F_q.$ \end{definition}

Note that if $A$ is the identity matrix, then $Q(x)=||x||.$ Hence, $Q(x)$ generalizes the distance function $||\cdot ||.$ 
For $E\subset \mathbb F_q^d, $ the non-degenerate quadratic distance set $\Delta_Q(E)$ is defined by
$$ \Delta_Q(E):=\{Q(x-y): x, y\in E\} \subset \mathbb F_q.$$
The quotient set of the quadratic distance set $\Delta_Q(E)$  is defined as
$$ \frac{\Delta_Q(E)}{\Delta_Q(E)}:= \left\{ \frac{b}{a}\in \mathbb F_q: a, b\in \Delta_Q(E), a\ne 0\right\}.$$
For $r\in \mathbb F_q^*,$ and $E\subset \mathbb F_q^d,$  define $W(r)$ as the number of quadruples $(x,y, z, w)\in E^4$ such that  $\frac{Q(x-y)}{Q(z-w)}=r,$ namely, 
\begin{equation}\label{DefW} W(r):=\sum_{\substack{x,y,z,w\in E:\\ \frac{Q(x-y)}{Q(z-w)}=r}} 1.\end{equation} 

Our main result is as follows.
\begin{theorem} \label{mainthm} For $E\subset \mathbb F_q^d, d\ge 2,$ let $W(r)$ be defined as in \eqref{DefW}. 
\begin{enumerate}
\item [(i)] If $d$ is even and $|E|\ge 4 q^{\frac{d}{2}},$ then  $W(r)\ge \dfrac{5|E|^4}{48q}$ for all $r$ in $\mathbb F_q^*.$
Moreover, the threshold $d/2$ is sharp for any finite field $\mathbb F_q$ and all even dimensions $d\ge 2.$

\item [(ii)] If $d$ is odd  and $|E|\ge 3 q^{\frac{d}{2}},$ then $W(r)\ge \dfrac{2|E|^4}{45q}$ for all non-zero square numbers $r$ in $\mathbb F_q^*.$
Furthermore, the threshold $d/2$ cannot be lower for any finite field $\mathbb{F}_q$ and all odd dimensions $d\geq 3$.  

\item [(iii)] If $d$ is odd and $|E|\ge \dfrac{11}{6} q^{\frac{d+1}{2}},$ then  $W(r)\ge \dfrac{2|E|^4}{363q}$ for all  $r$ in $\mathbb F_q^*.$  
In addition, the threshold $(d+1)/2$ cannot be improved for any finite field $\mathbb F_q$ and all odd dimensions $d\geq 3$.
\end{enumerate}
\end{theorem}  

As a consequence of Theorem \ref{mainthm}, we generalize Theorem \ref{IKPmthm} with  accurate constants improved for the cardinality of sets.
\begin{corollary}\label{maincor}  
Let $E \subset \mathbb F_q^d, d\ge 2,$ and $Q\in \mathbb F_q[x_1,\ldots, x_d]$ be a non-degenerate quadratic form. Then  the following statements are valid:
\begin{enumerate} \item  [(i)] If $d$ is even and $ |E|\ge 4 q^{d/2},$ then  $\dfrac{\Delta_Q(E)}{\Delta_Q(E)}=\mathbb F_q.$
Moreover, the exponent $d/2$ cannot be improved for arbitrary finite field $\mathbb F_q$ and  all even dimensions $d\ge2.$

\item [(ii)] If $d$ is odd and $|E|\ge 3 q^{d/2},$ then  $\dfrac{\Delta_Q(E)}{\Delta_Q(E)}\supseteq(\mathbb F_q)^2,$
where $(\mathbb F_q)^2 := \{a^2: a\in \mathbb F_q\}.$  In addition,  the exponent $d/2$ is optimal for arbitrary finite field $\mathbb{F}_q$ and all odd dimensions $d\geq 3$.

\item [(iii)]  If $d$ is odd and $|E|\ge \dfrac{11}{6} q^{(d+1)/2},$ then  $\dfrac{\Delta_Q(E)}{\Delta_Q(E)}=\mathbb F_q.$ Furthermore, the threshold $(d+1)/2$ cannot be lower for any finite field $\mathbb F_q$ and all odd dimensions $d\ge 3.$
\end{enumerate}
\end{corollary}
\begin{proof} We claim that the statements (i), (ii), (iii) of the corollary directly follow from those (i), (ii),(iii) of Theorem \ref{mainthm}, respectively. To prove our claim,  first observe that if $r\in \mathbb F_q^*,$  then  $W(r)>0$ if and only if  $r\in \frac{\Delta_Q(E)}{\Delta_Q(E)}.$ Combining this observation with Theorem \ref{mainthm} (i), (ii), (iii), the proof of the corollary is reduced to showing that 
$0\in   \frac{\Delta_Q(E)}{\Delta_Q(E)}$ under each assumption of Corollary (i), (ii), (iii). 
Since $ r\in \frac{\Delta_Q(E)}{\Delta_Q(E)} $ for some non-zero $r$ in $\mathbb F_q^*$,  we can choose $x,y, z, w$ in $E$ such that $Q(x-y)\ne 0$ and $Q(z-w)\ne 0.$ Hence, $0=\frac{Q(x-x)}{Q(z-w)} \in \frac{\Delta_Q(E)}{\Delta_Q(E)},$ as required.
\end{proof} 

We have the following few comments on our results, Theorem \ref{mainthm} and Corollary \ref{maincor}.
\begin{itemize}

\item  The results on Theorem \ref{mainthm} and Corollary \ref{maincor} are independent of the choice of non-degenerate quadratic forms $Q$.\\

\item Theorem \ref{mainthm} (i) improves and extends Theorem \ref{Phthm} to all even dimensions and arbitrary finite fields with the general non-degenerate quadratic distances. For instance, the theorem holds true and sharp without any further assumption such as the condition $q\equiv 3 \pmod{4},$ which was  necessary in proving Theorem \ref{Phthm} by using the group actions. In addition, our theorem provides explicit constants with  small numbers, which are based on our different approaches to estimate the counting functions.\\
 
\item Corollary \ref{maincor} (i), (ii) clearly generalize  Theorem \ref{IKPmthm} (i), (ii), respectively, with the improved constants. Notice that when $Q(x)=||x||$, the usual distance, the threshold $(d+1)/2$ in Corollary \ref{maincor} (iii) follows immediately from Theorem \ref{IRThm}. However, the constant $11/6$ for the set size in Corollary \ref{maincor} cannot be obtained from it since $2>11/6.$\\

\item In Section \ref{Sharpexamplesection}, we shall provide the sharpness examples for Theorem \ref{mainthm}. In particular, we shall show that the thresholds in the statements (i), (ii), (iii) of Theorem \ref{mainthm} are sharp for any finite field without any further assumptions on the size of the ground field $\mathbb F_q.$  In \cite{IKP}, some sharpness examples for Theorem \ref{IKPmthm} (i), (ii) were given but they work with some specific restriction to the size of $\mathbb F_q.$\\

\item In the proof of Theorem \ref{mainthm}, to efficiently estimate a lower bound of $W(r), r\ne 0,$ we will work on the product set $E\times E $ in $2d$-dimensions instead of the set $E \subset \mathbb F_q^d,$ and reduce our problem to the estimate of the explicit Fourier transform on quadratic homogeneous varieties in $2d$-dimensional vector spaces over $\mathbb F_q.$ This method enables us to obtain improved constants on the size conditions of sets (see, for example, Section \ref{ProofMthm}).\\

\end{itemize}

The rest of this paper will be organized to prove Theorem \ref{mainthm}.

\section{Equivalent forms of non-degenerate quadratic forms}
Let $Q(x)\in \mathbb F_q[x_1,\ldots, x_d]$ be a non-degenerate quadratic form.
Then we can write $Q(x)=x^\top A~x$ for some $d\times d$ matrix $A$ with $\det(A)\ne 0.$
Using a change of variables, by letting $x=Cy$ for some non-singular $d\times d$ matrix $C$, it follows that
$$ Q(x)= Q(Cy)=(Cy)^\top A~(Cy)= y^\top (C^\top A ~C)~y.$$
Letting $B=C^\top A ~C$, we obtain a non-degenerate quadratic form  $Q'(y):=Q(Cy)=y^\top B~y.$  In summary, after a non-singular change of variables,  the non-degenerate quadratic form $Q(x)$ can be transformed to another non-degenerate quadratic form $Q'(x).$ In this case, we say that $Q(x)$ is equivalent to $Q'(x).$ Furthermore, it is not hard to observe that 
$$ \Delta_Q(E)=\Delta_{Q'}(E'),$$
where $E\subset \mathbb F_q^d$ and $E':=\{C^{-1}x: x\in E\}.$ In addition, note that  $|E|=|E'|,$ and the size assumption of any sets $E$ is only considered as the main hypothesis for the distance type problems. Hence, without loss of generality, in the proof of Theorem \ref{mainthm} we may choose any non-degenerate quadratic form $Q'(x),$ which is equivalent to $Q(x),$  as a distance set. 

Now we introduce a concrete quadratic form $Q'(x)$ equivalent to any non-degenerate quadratic form $Q(x).$
Let $\eta$ denote the quadratic character of $\mathbb F_q^*,$ that is, $\eta(t)=1 $ if $t$ is a non-zero square number of $\mathbb F_q$, and $\eta(t)=-1$ if $t$ is a non-square number of $\mathbb F_q.$

\begin{lemma} [\cite{AM}, Theorem 1 and   \cite{Gr02}, P.79] \label{lemQ}Let $A$ be a $d\times d$ a symmetric matrix over $\mathbb F_q$ with $\det(A)\ne 0$ and it is associated with a non-degenerate quadratic form $Q(x)\in \mathbb F_q[x_1, \ldots, x_d].$ 
Then the following statements are true:
\begin{itemize}
\item[(i)] If $d$ is even, then the $Q(x)$ is equivalent to 
\begin{equation}\label{evenQ} Q'(x):= \sum_{i=1}^{d-1} (-1)^{i+1} x_i^2  - \varepsilon x_d^2=x_1^2-x_2^2+ \cdots + x_{d-1}^2-\varepsilon x_d^2,\end{equation}
where the $\varepsilon$ is taken as any element in $\mathbb F_q^*$ such that  $\eta((-1)^{\frac{d}{2}} \varepsilon)= \eta(\det(A)).$

\item [(ii)] If $d$ is odd, then the $Q(x)$ is equivalent to 
\begin{equation}\label{oddQ} Q'(x):= \sum_{i=1}^{d-1} (-1)^{i+1} x_i^2  + \varepsilon x_d^2=x_1^2-x_2^2+ \cdots + x_{d-2}^2-x_{d-1}^2 + \varepsilon x_d^2,\end{equation}
where the $\varepsilon$ is taken as any element in $\mathbb F_q^*$ such that  $\eta((-1)^{\frac{d-1}{2}} \varepsilon)= \eta(\det(A)).$
\end{itemize} \end{lemma}
In the above lemma,  note that if $\eta(\varepsilon)=1,$ we can simply choose $\varepsilon=1.$ On the other hand, if $\eta(-\varepsilon)=1,$ then we can take $-\varepsilon =1$, namely, $\varepsilon =-1.$

\begin{example}
When $Q(x)=||x||$, it is clear that $\det(A)=1.$ 
If $d\equiv 2 \pmod{4},$ then, by Lemma \ref{lemQ}  (i),  $\eta(-\varepsilon)=1$ and so $Q(x)=||x||$ can be transformed to the following form:
$$ x_1^2-x_2^2+ \cdots+ x_{d-1}^2+x_d^2.$$ 
If $d\equiv 0 \pmod{4},$ then, by Lemma \ref{lemQ} (i),  $\eta(\varepsilon)=1$ and so $Q(x)=||x||$ can be transformed to the form below:
$$ x_1^2-x_2^2 + \cdots + x_{d-1}^2 -x_d^2.$$
On the other hand, when $d$ is odd, we can apply Lemma \ref{lemQ} (ii) to conclude that $Q(x)=||x||$ can be transformed to the quadratic form $x_1^2-x_2^2+ \cdots + x_{d-2}^2-x_{d-1}^2 - x_d^2$ for $d\equiv 3 \pmod{4},$  and to the quadratic form $x_1^2-x_2^2+ \cdots + x_{d-2}^2-x_{d-1}^2 + x_d^2$ for $d\equiv 1 \pmod{4}.$ 
\end{example}
For instance, we can consider the following simple, concrete example. 
\begin{example} 
Let $Q(x)=||x||\in \mathbb F_3[x_1, x_2, x_3].$ Then, by the above example,  we know that $Q(x)=x_1^2+x_2^2+x_3^2$ is equivalent to $x_1^2-x_2^2-x_3^2.$ 
Now let us prove this fact by a direct non-singular linear substitution.  To do this,  we note that  
$Q(x)=x^\top I_3 ~x,$ where $I_3$ denotes the $3\times 3$ identity matrix.  We use a change of variables, by letting $x=Cy$, where $C$ is the $3\times 3$ non-singular symmetric matrix defined as below:
$$ C= \left(\begin{matrix}
1&0&0\\
0&1&1\\
0&1&-1
\end{matrix}\right)=  C^\top.$$
It follows that 
$$ Q(x)=Q(Cy)=y^\top (C^\top I_3 C) y= y^\top (C^\top C) y=y^\top\left(\begin{matrix}
1&0&0\\
0&2&0\\
0&0&2
\end{matrix}\right)y.$$
Since $2\equiv -1 \pmod{3},$  we conclude that $Q(x)=||x||$ is equivalent to $x_1^2-x_2^2-x_3^2,$ as required.
\end{example}

As mentioned in the beginning of this section,  we may assume that  any non-degenerate quadratic form $Q(x)\in \mathbb F_q[x_1, \ldots, x_d]$ can be identified with  $Q'(x)$ defined as in Lemma \ref{lemQ} (i), (ii). Thus, from now on, we fix the definition of the non-degenerate quadratic form $Q(x)$, which we will use as the standard distance for the non-degenerate quadratic form $Q(x)\in \mathbb F_q[x_1, \ldots, x_d].$ 

\begin{definition} \label{defnormQ} (Standard distance functions and its dual functions) Let $Q(x)\in \mathbb F_q[x_1, \ldots, x_d]$ be a non-degenerate quadratic form and $A$ be its associated matrix. 
Then we define the standard distance function $||\cdot ||_Q$  and its dual function $||\cdot||_{Q^*}$ on $\mathbb F_q^d$ as follows:
\begin{itemize}  \item [(i)] If $d$ is even, then 
$$ ||x||_Q:=\sum_{i=1}^{d-1} (-1)^{i+1} x_i^2  - \varepsilon x_d^2=x_1^2-x_2^2+ \cdots + x_{d-1}^2-\varepsilon x_d^2,$$
$$ ||m||_{Q^*}:=\sum_{i=1}^{d-1} (-1)^{i+1} m_i^2  -\varepsilon^{-1}  m_d^2= m_1^2- m_2^2+ \cdots +  m_{d-1}^2-\varepsilon^{-1} m_d^2,$$
where the $\varepsilon$ can be taken as any element in $\mathbb F_q^*$ such that  $\eta((-1)^{\frac{d}{2}} \varepsilon)= \eta(\det(A)).$
\item [(ii)] If $d$ is odd, then 
$$ ||x||_Q:= \sum_{i=1}^{d-1} (-1)^{i+1} x_i^2  + \varepsilon x_d^2=x_1^2-x_2^2+ \cdots + x_{d-2}^2-x_{d-1}^2 + \varepsilon x_d^2,$$
$$||m||_{Q^*}:= \sum_{i=1}^{d-1} (-1)^{i+1}  m_i^2  + \varepsilon^{-1} m_d^2= m_1^2- m_2^2+ \cdots +  m_{d-2}^2- m_{d-1}^2 + \varepsilon^{-1} m_d^2,$$
where the $\varepsilon$ is taken as any element in $\mathbb F_q^*$ such that  $\eta((-1)^{\frac{d-1}{2}} \varepsilon)= \eta(\det(A)).$
\end{itemize}
\end{definition}

We also define the spheres with respect to the distances $||\cdot||_{Q}$ and $||\cdot||_{Q^*}.$
\begin{definition} \label{defSQt} Let $t\in \mathbb F_q$ and $Q(x) \in \mathbb F_q[x_1, \ldots, x_d]$ be a non-degenerate quadratic form.  We define
$$(S_{Q})_t:=\{x\in \mathbb F_q^d: ||x||_Q=t\} \quad\mbox{and}\quad (S_{Q^*})_t:=\{m\in \mathbb F_q^d:  ||m||_{Q^*}=t\}.$$
\end{definition}
The variety $(S_{Q})_t$ can be regarded as  a sphere of radius $t$ with the standard distance function $||\cdot||_Q.$
The variety $(S_{Q^*})_t$ can be called the dual sphere of $(S_{Q})_t.$

\section{Basics on the discrete Fourier analysis and related lemmas}

The discrete Fourier analysis machinery will function as a main tool in proving Theorem \ref{mainthm}.
In this section,  we introduce some basics on it without proofs and  review main properties of the Gauss sum. 
In addition, we deduce a general counting lemma which will play a key role in the proof of our main theorem, Theorem \ref{mainthm}.

\subsection{Discrete Fourier analysis and  Gauss sums}

We begin by defining the canonical additive character of $\mathbb F_q$.
Let $q$ be a power of prime $p$, say that $q=p^\ell.$  
The absolute trace function  $\text{Tr}: \mathbb F_q \to \mathbb F_p$ is well defined as  
$ \text{Tr}(t)=t+ t^p+ t^{p^2} + \cdots + t^{p^{\ell-1}}$ (for example, see Section 3  of \cite{LN97}).

\begin{definition} (Canonical additive character and the quadratic character, \cite{LN97}) \label{def1}
 The function $\chi$ defined by
 $$\chi(c)=e^{2\pi i \text{Tr}(c)/p} \quad \mbox{for}~~ c\in \mathbb F_q$$
 is called the canonical additive character of $\mathbb F_q.$  On the other hand,  the multiplicative character $\eta$ is a function $\mathbb F_q^* \to \mathbb{R},$ defined by
 $$ \eta(t)=\left\{ \begin{array}{ll} 1 \quad &\mbox{if}~t ~\mbox{is a square number of}~ \mathbb F_q^*,\\
                                                   -1 \quad &\mbox{if}~ t ~\mbox{is not a square number of }\mathbb F_q^*.\end{array}\right.$$

\end{definition}

Recall  that  $ \eta(-1)=-1 \iff q\equiv 3 \pmod{4},$  and   $\eta(-1)=1 \iff q\equiv 1 \pmod{4}$ (Remark 5.13, \cite{LN97}).

Both the additive character $\chi$ and the multiplicative character $\psi$ enjoy the orthogonality property: For any $m\in \mathbb F_q^n, ~n\ge 1$,
$$ \sum_{x\in \mathbb F_q^n} \chi(m\cdot x)=\left\{ \begin{array}{ll} 0  
\quad&\mbox{if}~~ m\neq (0,\dots,0)\\
q^n  \quad &\mbox{if} ~~m= (0,\dots,0), \end{array}\right.
\quad \mbox{and}\quad \sum_{t\in \mathbb F_q^*} \eta(at)=0  \quad\mbox{if}~~ a\ne 0.$$
where  $m\cdot x$ denotes the usual dot product notation.

For a function $f: \mathbb F_q^n \to \mathbb C,$ the Fourier transform of $f$ is defined as 
$$ \widehat{f}(m):=q^{-n} \sum_{x\in \mathbb F_q^n} \chi(-m\cdot x) f(x)$$
and the Fourier inversion theorem states that 
$$f(x)=\sum_{m\in \mathbb F_q^n} \chi(m\cdot x) \widehat{f}(m).$$

The Plancherel theorem in this context says that
$$ \sum_{m\in \mathbb F_q^n} |\widehat{f}(m)|^2 = \frac{1}{q^n} \sum_{x\in \mathbb F_q^n} |f(x)|^2.$$
In particular, for any set $E \subset \mathbb F_q^n,$ we have
\begin{equation}\label{basicP} \sum_{m\in \mathbb F_q^n} |\widehat{E}(m)|^2 = \frac{|E|}{q^n}= \widehat{E}(0,\ldots,0).\end{equation}
Here, and throughout this paper, we identify the set $E \subset \mathbb F_q^d$ with the indicator function $1_E$ of the set $E.$ In particular,  when $E=(0, \dots, 0),$  we write 
$\delta_0(x)$ for the indicator function $1_E(x).$\\

We now introduce Gauss sums.
\begin{definition}  Let $\chi, \eta$ denote the canonical additive character and the quadratic character of $\mathbb F_q^*.$ The standard Gauss sum determined by $\chi$ and $\eta$  is defined by 
$$\mathcal{G}=\mathcal{G}(\eta, \chi):= \sum_{t\in \mathbb F_q^*} \eta(t) \chi(t).$$

\end{definition}

The following estimate is well-known (see \cite{LN97}):
$$|\mathcal{G}|=\sqrt{q}.$$ 

The following theorem is very useful in finding an explicit value of the Fourier transform on the homogeneous variety.  
 \begin{lemma}\label{SGa} Let $\mathcal{G}$ denote the standard Gauss sum. Then we have
 $$ \mathcal{G}^2=\mathcal{G}(\eta, \chi)^2=\eta(-1) q.$$
 \end{lemma}
\begin{proof}
It is obvious that  $\eta =\overline{\eta}$ and $\overline{\chi(t)}= \chi(-t).$ Therefore, it is seen by a change of variables that
$$ \mathcal{G}(\eta, \chi)=\eta(-1) \overline{\mathcal{G}(\eta, \chi)}.$$ 
Hence,  $\mathcal{G}(\eta, \chi)^2= \eta(-1) |\mathcal{G}(\eta, \chi)|^2.$  Since $|\mathcal{G}(\eta, \chi)|=\sqrt{q},$ we are done.
\end{proof}

It is not hard to note that  for any  $a\in \mathbb F_q^*,$ 
$$ \sum_{t\in \mathbb F_q} \chi(at^2)=\eta(a) \mathcal{G}.$$
Completing the square and using a simple change of variables,  the above formula can be generalized to the formula below:
For any  $a\in \mathbb F_q^*$ and  any $b\in \mathbb F_q,$   we have
\begin{equation}\label{CSQ} \sum_{t\in \mathbb F_q} \chi(at^2+bt) = \eta(a) \chi\left(\frac{b^2}{-4a}\right)\mathcal{G} .\end{equation}

By a change of variables and properties of the quadratic character $\eta,$  it is not difficult to note that 
for $b\ne 0$, we have
\begin{equation}\label{Gauss}\sum_{s\in \mathbb F_q^*} \eta(s) \chi(bs^{-1})=\sum_{s\in \mathbb F_q^*} \eta(bs^{-1}) \chi(s) =\sum_{s\in \mathbb F_q^*} \eta(bs) \chi(s)=\eta(b) \mathcal{G}. 
\end{equation}

\subsection{General counting lemma}
In order to prove Theorem \ref{IRthm}, the first listed author and Rudnev  \cite{IR07} evaluated the values of the counting function $v(t)$ in \eqref{countingF} by relating it to the Fourier transform on $S_t:=\{x\in \mathbb F_q^d: ||x||=t\},$ the sphere of radius $t\in \mathbb F_q^*.$  In this subsection, we formulate their work in the general setting, which will provide an initial step in estimating  $W(r)$ in \eqref{DefW}.

Now let us work on $\mathbb F_q^n$ for an integer $n\ge 2.$
\begin{lemma} [General counting lemma] \label{GCL}
Let $P(\mathbf{x})$ be a polynomial function on $\mathbb F_q^n.$ Consider an algebraic variety 
$V:=\{\mathbf{x}\in \mathbb F_q^n: P(\mathbf{x})=0\}.$ 
Then, for every set $\mathcal{E} \subset \mathbb F_q^n,$ we have
$$ \sum_{\substack{\mathbf{x}, \mathbf{y} \in \mathcal{E}:\\
 P(\mathbf{x}-\mathbf{y})=0}} 1 = q^{2n} \sum_{\mathbf{m}\in \mathbb F_q^n} \widehat{V}(\mathbf{m}) ~|\widehat{\mathcal{E}}({\mathbf m})|^2.$$
\end{lemma}
\begin{proof} 
The proof proceeds by modifying the method of the first listed author and Rudnev \cite{IR07}. 

It follows that
$$ \mathbf{w}(0):=\sum_{\substack{\mathbf{x}, \mathbf{y} \in \mathcal{E}:\\ P(\mathbf{x}-\mathbf{y})=0}} 1=\sum_{\mathbf{x}, \mathbf{y} \in \mathbb F_q^n } V(\mathbf{x}-\mathbf{y}) \mathcal{E}(\mathbf{x}) \mathcal{E}(\mathbf{y}).$$
Applying the Fourier inversion theorem to the indicator function $V(\mathbf{x}-\mathbf{y})$, we get
$$ \mathbf{w}(0)=\sum_{\mathbf{x}, \mathbf{y} \in \mathbb F_q^n } \sum_{\mathbf{m}\in \mathbb F_q^n} \widehat{V}(\mathbf{m}) \chi(\mathbf{m}\cdot (\mathbf{x}-\mathbf{y})) \mathcal{E}(\mathbf{x}) \mathcal{E}(\mathbf{y}).$$
 Finally, applying the definition of the Fourier transform $\widehat{\mathcal{E}}(\mathbf{m}),$   the statement follows.
\end{proof}

The following is a direct consequence of the general counting lemma.
\begin{corollary} Let $Q(x)=||x||_Q$ for $x\in \mathbb F_q^d.$
Then, for every set $E\subset \mathbb F_q^d,$ we have
$$ w(0):=\sum_{\substack{x,y\in E:\\ Q(x-y)=0}} 1 =q^{2d} \sum_{m\in \mathbb F_q^d} \widehat{(S_Q)_0}(m)~ |\widehat{E}(m)|^2.$$
\end{corollary}
\begin{proof} By Definition \ref{defSQt}, recall that
$$ (S_Q)_0=\{x\in \mathbb F_q^d: Q(x)=0\}.$$
Then the statement of the corollary  immediately follows, because it is a special case of Lemma \ref{GCL} when $n=d, V=(S_Q)_0,  \mathcal{E}=E,$ and $P=Q.$
\end{proof}
\section{Sharpness of Theorem \ref{mainthm}}\label{Sharpexamplesection}
In this section, we provide  sharpness examples for the threshold results on  Theorem \ref{mainthm} (i), (ii), (iii).
In order to construct such sets, it is enough to use the standard quadratic forms in Lemma \ref{lemQ} or Definition \ref{defnormQ}. This is because any non-degenerate quadratic forms that are equivalent yield  the same result on the distance type problem. 

\subsection{The proof of sharpness of Theorem \ref{mainthm} (i)}
Let the dimension $d$ be even.  When $d\ge 4,$  we define a set $E_1 \subset \mathbb F_q^{d-2}$ as 
$$ E_1:=\left\{(t_1, t_1, \ldots, t_i, t_i, \ldots, t_{(d-2)/2}, t_{(d-2)/2}) \in \mathbb F_q^{d-2}:  t_i\in \mathbb F_q,  1\le i\le (d-2)/2\right\}.$$
Now if $d\ge 4$ is even, then define $E= E_1\times \mathbb F_q\times \{0\}=\{(x', x_{d-1}, 0)\in \mathbb F_q^d:  x'\in E_1,  x_{d-1}\in \mathbb F_q\},$ and if $d=2,$ then define $E=\mathbb F_q\times \{0\}.$
Since $d$ is even,   the standard distance function is given by the form
$$Q(x)=||x||_Q= x_1^2-x_2^2 + \cdots + x_{d-1}^2 - \varepsilon x_d^2.$$  
Now notice that
$ |E|=|E_1| |\mathbb F_q|= q^{d/2}$ and  $\Delta_Q(E)=\{(a-b)^2: a, b\in \mathbb F_q\} = (\mathbb F_q)^2.$ 
Since the ratio of non-zero square numbers is also a square number in $\mathbb F_q$,  it follows that $\frac{\Delta_Q(E)}{\Delta_Q(E)}=(\mathbb F_q)^2.$ 
Since $(\mathbb F_q)^2  \subsetneq \mathbb F_q$,  there is $r\in \mathbb F_q^*$ and a set $E\subset \mathbb F_q^d$ with $|E|=q^{d/2}$  such that $r\notin \frac{\Delta_Q(E)}{\Delta_Q(E)},$ namely, $W(r)=0.$

In short, the set $E$ provides the desired sharpness example.
Notice that, in our construction of the set $E$, we did not impose any specific assumptions on the underlying finite field $\mathbb F_q.$

\subsection{The proof of sharpness of Theorem \ref{mainthm} (ii)}
Since $d\ge 3$ is odd, we adopt the following standard distance function in Definition \ref{defnormQ}:
\begin{equation}\label{StaN}Q(x)=||x||_Q= x_1^2-x_2^2 + \cdots + x_{d-2}^2-x_{d-1}^2  + \varepsilon x_d^2.\end{equation}
We define a set $H \subset \mathbb F_q^{d-1}$ as 
\begin{equation}\label{defHK} H:=\left\{(t_1, t_1, \ldots, t_i, t_i, \ldots, t_{(d-1)/2}, t_{(d-1)/2}) \in \mathbb F_q^{d-1}: t_i\in \mathbb F_q,  1\le i\le (d-1)/2\right\}.\end{equation}
It is clear that $|H|=q^{(d-1)/2}$ and  the Cartesian product of two sets $H$ and $A\subset \mathbb F_q$ is defined by $H\times A:=\{(x',x_{d})\in \mathbb F_q^d: x'\in H,\ x_d\in A\}.$ Letting $E=H\times A,$  we see that  $\Delta_Q(E)=\{\varepsilon (a-b)^2: a, b\in A\}$ and so 
$ \frac{\Delta_Q(E)}{\Delta_Q(E)}=\left\{\frac{(a-b)^2}{(a'-b')^2}\in \mathbb F_q: a, b, a',b'\in A,\  a'\ne b'\right\} \subset (\mathbb F_q)^2.$

\smallskip

Let $q=p^{\ell}$ for an integer $\ell \geq 1$ and  an odd prime $p.$ Then one can consider the finite field $\mathbb{F}_q$ as an $\ell$-dimensional vector space over $\mathbb{F}_p$ with a basis $\left\{1,\theta,\dots,\theta^{\ell-1}\right\}$, where $\theta$ is an algebraic element of degree $\ell$ over $\mathbb{F}_p$. For every $0<\delta<1/2$, we can choose an arithmetic progression $B_{\delta}\subset \mathbb{F}_p$ with $|B_{\delta}|\sim p^{1/2-\delta}$. Let us define the set $A_{\delta}:=\left\{b_0+b_1\theta+\dots+b_{\ell-1}\theta^{\ell-1}: b_0,\dots,b_{\ell-1}\in B_{\delta}\right\}\subset \mathbb{F}_q$ and it is easy to verify that $|A_{\delta}|=|B_{\delta}|^{\ell}\sim p^{\ell/2-\delta\ell}=q^{1/2-\delta}$. If we consider the difference set of $A_{\delta}$ which is defined as $$A_{\delta}-A_{\delta}:=\left\{(b_0-b_0')+(b_1-b_1')\theta+\dots+(b_{\ell-1}-b_{\ell-1}')\theta^{\ell-1}: b_0,b_0',\dots,b_{\ell-1},b_{\ell-1}'\in B_{\delta}\right\},$$ then  $|A_{\delta}-A_{\delta}|=|B_{\delta}-B_{\delta}|^{\ell}\sim |B_{\delta}|^{\ell}\sim p^{\ell/2-\delta\ell}$. Here we have used the fact that $|B_{\delta}-B_{\delta}|\sim |B_{\delta}|$ since $B_{\delta}$ is an arithmetic progression. Therefore, we have shown that $|A_{\delta}-A_{\delta}|\sim |A_{\delta}|$. 

\smallskip

Set $E_{\delta}=H\times A_{\delta}\subset \mathbb{F}_q^d$. Then $|E_{\delta}|=|H||A_{\delta}|\sim q^{\frac{d}{2}-\delta}$. Observe that $\Delta_Q(E_{\delta})=\left\{\varepsilon(a-a')^2\in \mathbb{F}_q: a,a'\in A_{\delta}\right\}$. Since $|A_{\delta}-A_{\delta}|\sim |A_{\delta}|$, we see that $|\Delta_Q(E_{\delta})|\sim |A_{\delta}|$. Since $\left|\frac{\Delta_Q(E_{\delta})}{\Delta_Q(E_{\delta})}\right|\leq |\Delta_Q(E_{\delta})|^2\sim q^{1-2\delta}$ and $|(\mathbb{F}_q)^2|=\frac{q+1}{2}$, we conclude that $\frac{\Delta_Q(E_{\delta})}{\Delta_Q(E_{\delta})}\subsetneq (\mathbb{F}_q)^2$. Thus, for every $0<\delta<1/2$, there exists $E_{\delta}\subset \mathbb{F}_q^d$ with $|E_{\delta}|\sim q^{\frac{d}{2}-\delta}$ such that $W(r)=0$ for some non-zero $r\in (\mathbb{F}_q)^2$. Since $\delta$ can be taken as an arbitrary small positive number, the threshold $d/2$ cannot be smaller.

\subsection{The proof of sharpness of Theorem \ref{mainthm} (iii)}
Suppose that $d$ is odd.  Then, we can also use the standard distance function in \eqref{StaN}.
Let $H\subset \mathbb F_q^{d-1}$ be the set defined as in \eqref{defHK}.

Set $E:= H\times \mathbb F_q\subset \mathbb F_q^d.$  Then it is not hard to see that
$ |E|=q^{\frac{d+1}{2}}$ and $\Delta_Q(E)=\{\varepsilon (a-a')^2: a, a'\in \mathbb F_q\}.$
Also notice that 
$$\frac{\Delta_Q(E)}{\Delta_Q(E)}=\left\{\frac{ (a-a')^2}{(b-b')^2}: a,a', b, b'\in \mathbb F_q, b\ne b'\right\}=(\mathbb F_q)^2.$$
Clearly, there is  $r\in \mathbb F_q^*$ such that  $r\notin \frac{\Delta_Q(E)}{\Delta_Q(E)},$ i.e., $W(r)=0.$
Thus, the set $E$ guarantees the sharpness of Theorem \ref{mainthm} (iii).

\section{Proof of the main theorem (Theorem \ref{mainthm})} \label{ProofMthm}
In this section we start proving Theorem \ref{mainthm}. We may assume that $Q(x)=||x||_{Q}$ and $Q^*(m)=||m||_{Q^*}$ for $x, m\in \mathbb F_q^d,$ where $||x||_{Q}$ and $||m||_{Q^*}$ are defined as in Definition \ref{defnormQ}.
Let $E\subset \mathbb F_q^d.$ For each $r\in \mathbb F_q^*,$ we begin with estimating  $W(r)$ defined as in \eqref{DefW}.
First observe that  $W(r)$ can be written as
$$ W(r)=\sum_{\substack{x,y, z, w\in E:\\ Q(x-y)=rQ(z-w)}}1 - \sum_{\substack{x,y, z, w\in E:\\ Q(x-y)=0=Q(z-w)}} 1=: M(r)-w^2(0),$$ 
where $w(0)$ is defined by
\begin{equation}\label{defw0} w(0)=\sum_{\substack{\alpha, \beta\in E:\\ Q(\alpha-\beta)=0}} 1.\end{equation}

To estimate $M(r),$ the first sum above, we write it as
$$ M(r)=\sum_{\substack{(x,z), (y,w) \in E\times E:\\ Q(x-y)-rQ(z-w)=0}} 1.$$
We now relate $M(r)$ to an estimate on a homogeneous variety of degree two in $\mathbb F_q^{2d}.$
To this end, we need to introduce the following definition.
\begin{definition} 
Let $r\in \mathbb F_q^*, Q(x)=||x||_Q,$ and $Q^*(m)=||m||_{Q^*}$ for $x, m\in \mathbb F_q^d.$   Then, for $x, x', m, m'\in \mathbb F_q^d,$  we define 
$$ \mathcal{Q}_r (x, x'):= Q(x)-rQ(x') \quad \mbox{and}\quad \mathcal{Q}_r^*(m, m'):=Q^*(m)-r^{-1}Q^*(m').$$
In addition,  we define algebraic varieties $V_{\mathcal{Q}_r} $ and $V_{\mathcal{Q}_r^*}$ lying in $\mathbb F_q^{2d}$ as follows:
$$ V_{\mathcal{Q}_r}:= \{X\in \mathbb F_q^{2d}: \mathcal{Q}_r(X)=0\} \quad \mbox{and}\quad V_{\mathcal{Q}_r^*}:=\{M\in \mathbb F_q^{2d}: \mathcal{Q}_r^*(M)=0\}.$$
\end{definition}

Letting $X=(x,z), Y=(y, w)\in E\times E$ and using the notation in the above definition, we have
\begin{equation}\label{equK1} M(r)= \sum_{\substack{X, Y\in E\times E:\\ \mathcal{Q}_r(X-Y)=0}}1=q^{4d} \sum_{M\in \mathbb F_q^{2d}} \widehat{V_{\mathcal{Q}_r}}(M) |\widehat{E\times E} (M)|^2,\end{equation}
where the last equality follows from the general counting lemma (see Lemma \ref{GCL}).
We will invoke the following explicit formula for $\widehat{V_{\mathcal{Q}_r}}(M),$ whose proof will be given in the following section. For a moment, we accept the proposition below.
\begin{proposition} \label{propFV} For $x=(x_1, \ldots, x_d), m=(m_1, \ldots, m_d)\in \mathbb F_q^d,$ let
$Q(x)=||x||_Q$ and $Q^*(m)=||m||_{Q^*},$  where $||x||_Q$ and $||m||_{Q^*}$ are defined as in Definition \ref{defnormQ}. 
Then, for each $r\in \mathbb F_q^*$ and $M\in \mathbb F_q^{2d},$  the Fourier transform $\widehat{V_{\mathcal{Q}_r}}(M)$ of the indicator function of the variety $V_{\mathcal{Q}_r}$ in $\mathbb F_q^{2d}$ is as follows:
\begin{itemize}
\item [(i)] Let $d$ be even. Then we have
$$ \widehat{V_{\mathcal{Q}_r}}(M) = \frac{\delta_0(M)}{q} + \frac{V_{\mathcal{Q}_r^*}(M)}{q^d} -\frac{1}{q^{d+1}}.$$

\item [(ii)] Let $d$ be odd. Then we have
$$ \widehat{V_{\mathcal{Q}_r}}(M) = \frac{\delta_0(M)}{q} + \frac{\eta(r)V_{\mathcal{Q}_r^*}(M)}{q^d}-\frac{\eta(r)}{q^{d+1}} .$$
\end{itemize}
\end{proposition}
After replacing $\widehat{V_{\mathcal{Q}_r}}(M)$ in \eqref{equK1} by its explicit value given in Proposition \ref{propFV}, we can easily get 
$$ M(r)=\frac{|E|^4}{q} + q^{3d} \sum_{M\in V_{\mathcal{Q}_r^*}} |\widehat{E\times E}(M)|^2 -q^{d-1}|E|^2 \quad\mbox{for}~d\ge 2~\mbox{even},$$
and
$$ M(r)=\frac{|E|^4}{q} + q^{3d} \eta(r) \sum_{M\in V_{\mathcal{Q}_r^*}} |\widehat{E\times E}(M)|^2 -q^{d-1}\eta(r)|E|^2 \quad\mbox{for}~d\ge 3~\mbox{odd}.$$

We also need the following proposition to estimate the quantity $w(0)$ in \eqref{defw0}. 
\begin{proposition} \label{lem2.3K} Let $E\subset \mathbb F_q^d.$ For $t\in \mathbb F_q$, we define $w(t)$ as the number of pairs $(x,y)\in E\times E$ such that $||x-y||_Q=t,$ namely, 
$$ w(t):=\sum_{\substack{x,y\in E:\\ ||x-y||_Q=t}}1,$$
where  $||\cdot||_{Q}$ is the standard distance function in Definition \ref{defnormQ}. Then the following statements hold true:

\begin{itemize}
\item [(i)] If $d$ is even and $\eta\left(\varepsilon \right)=1,$ then we have 
$$ 0\le w(0) \le  \frac{|E|^2}{q} +  q^{\frac{3d}{2}} \sum_{m\in (S_{Q^*})_0} |\widehat{E}(m)|^2.$$

\item [(ii)] If $d$ is even and $\eta\left(\varepsilon \right)=-1,$ then we have 
$$ 0\le w(0) \le  \frac{|E|^2}{q} +  q^{\frac{d-2}{2}} |E|.$$

\item [(iii)] If $d$ is odd, then we have
$$ 0\le w(0) \le  \frac{|E|^2}{q} +  q^{\frac{d-1}{2}} |E|.$$
\end{itemize}
\end{proposition}
We postpone the proof of Proposition \ref{lem2.3K} to the following section and let us accept it for a moment. 

Since $W(r)=M(r)-w^2(0),$  invoking (i),(ii),(iii) of Proposition \ref{lem2.3K} together with the estimates of $M(r)$, lower bounds of $W(r)$  are obtained as follows:
\begin{itemize}
\item [(Case A)] If $d\ge 2$ is even and $\eta(\varepsilon)=1,$ then 
\begin{equation}\label{CaseA} W(r)\ge \frac{|E|^4}{q} + q^{3d} \sum_{M\in V_{\mathcal{Q}_r^*}} |\widehat{E\times E}(M)|^2 -q^{d-1}|E|^2 - \left(\frac{|E|^2}{q} +  q^{\frac{3d}{2}} \sum_{m\in (S_{Q^*})_0} |\widehat{E}(m)|^2\right)^2.\end{equation}
\item [(Case B)] If $d\ge 2$ is even and $\eta(\varepsilon)=-1,$ then we have
\begin{equation}\label{CaseB}W(r)\ge \frac{|E|^4}{q} + q^{3d} \sum_{M\in V_{\mathcal{Q}_r^*}} |\widehat{E\times E}(M)|^2 -q^{d-1}|E|^2 -\left(\frac{|E|^2}{q} +  q^{\frac{d-2}{2}} |E|   \right)^2.\end{equation}
\item [(Case C)] If $d\ge 3$ is odd, then we have
\begin{equation}\label{CaseC} W(r)\ge \frac{|E|^4}{q} + q^{3d} \eta(r) \sum_{M\in V_{\mathcal{Q}_r^*}} |\widehat{E\times E}(M)|^2 -q^{d-1}\eta(r)|E|^2 -\left(\frac{|E|^2}{q} +  q^{\frac{d-1}{2}} |E|   \right)^2.\end{equation}
\end{itemize}

In the following subsections, we will complete the proof of Theorem \ref{mainthm} (i), (ii), (iii), by assuming that Propositions \ref{propFV} and \ref{lem2.3K} hold true.

\subsection{Proof of Theorem \ref{mainthm} (i)} Let $d\ge 2$ be even. Suppose that $E\subset \mathbb F_q^d$ with $|E|\ge 4 q^{d/2}.$ It is clear that  $\eta(\varepsilon)$ is either $1$ or $-1.$

First, we find a lower bound of $W(r)$ under the assumption of (Case A).
To do this, we  expand the last term in \eqref{CaseA} and observe that
$$ q^{3d} \sum_{M\in V_{\mathcal{Q}_r^*}} |\widehat{E\times E}(M)|^2 - \left(q^{\frac{3d}{2}} \sum_{m\in (S_{Q^*})_0} |\widehat{E}(m)|^2\right)^2 \ge 0$$ 
and $$ \sum_{m\in (S_{Q^*})_0} |\widehat{E}(m)|^2 \le \sum_{m\in \mathbb F_q^d} |\widehat{E}(m)|^2 =\frac{|E|}{q^d}.$$
Combining these observations and the inequality \eqref{CaseA},  we get
$$ W(r)\ge  \frac{|E|^4}{q} - q^{d-1}|E|^2 -\frac{|E|^4}{q^2} -2 q^{\frac{d-2}{2}}|E|^3.$$
Since $1=\frac{5}{48} + \frac{3}{48} +\frac{16}{48} +\frac{24}{48},$ we can write
$\frac{|E|^4}{q}= \frac{5}{48} \frac{|E|^4}{q}+ \frac{3}{48} \frac{|E|^4}{q}+\frac{16}{48}\frac{|E|^4}{q} +\frac{24}{48}\frac{|E|^4}{q}.$ Hence, the lower bound of $W(r)$ can be written as
$$ W(r)\ge \frac{5}{48} \frac{|E|^4}{q} + \left(\frac{3}{48} \frac{|E|^4}{q}-q^{d-1}|E|^2\right) 
+\left(\frac{16}{48}\frac{|E|^4}{q}-\frac{|E|^4}{q^2}\right) +\left(\frac{24}{48}\frac{|E|^4}{q}-2 q^{\frac{d-2}{2}}|E|^3\right).$$ 
Since $q\ge 3$ and $|E|\ge 4 q^{d/2},$  each value in parentheses above is non-negative. Hence, we obtain the desired result:
$$ W(r)\ge \frac{5}{48} \frac{|E|^4}{q}.$$

Next, let us find a lower bound of $W(r)$ under the assumption of (Case B).
We also expand the last term in \eqref{CaseB} and then observe that the sum in \eqref{CaseB} can be ignored since it is positive.
Thus, we get
$$ W(r)\ge \frac{|E|^4}{q} - q^{d-1}|E|^2 -\frac{|E|^4}{q^2} -q^{d-2}|E|^2-2 q^{\frac{d-4}{2}}|E|^3.$$
Since $1=\frac{20}{48} + \frac{3}{48} +\frac{16}{48} +\frac{1}{48}+\frac{8}{48},$  we can write
$$ W(r)\ge \frac{20}{48}\frac{|E|^4}{q} + \left( \frac{3}{48}\frac{|E|^4}{q} - q^{d-1}|E|^2\right) + \left(\frac{16}{48}\frac{|E|^4}{q}-\frac{|E|^4}{q^2}\right)+ \left(\frac{1}{48}\frac{|E|^4}{q}-q^{d-2}|E|^2\right)+ \left(\frac{8}{48}\frac{|E|^4}{q}-2 q^{\frac{d-4}{2}}|E|^3\right).$$
Since $q\ge 3$ and $|E|\ge 4q^{d/2}$,  it is not hard to notice that  each value in parentheses above is non-negative. Hence, in this case we get a better lower bound:
$$ W(r)\ge \frac{20}{48}\frac{|E|^4}{q} \ge \frac{5}{48} \frac{|E|^4}{q} .$$
This completes  the proof of the first part of Theorem \ref{mainthm}.

\subsection{Proof of Theorem \ref{mainthm} (ii)} 
Let $d\ge 3$ be odd and $E\subset \mathbb F_q^d.$ 
Assume that $r\in \mathbb F_q^*$ is a non-zero square number. Then $\eta(r)=1.$ Thus, 
it follows from \eqref{CaseC} of (Case C) that 
\begin{equation*}\label{WrK} W(r)\ge \frac{|E|^4}{q} + q^{3d}  \sum_{M\in V_{\mathcal{Q}_r^*}} |\widehat{E\times E}(M)|^2 -q^{d-1}|E|^2 -\left(\frac{|E|^2}{q} +  q^{\frac{d-1}{2}} |E|   \right)^2.\end{equation*}
Expand the above square term and notice that
$$  \sum_{M\in V_{\mathcal{Q}_r^*}} |\widehat{E\times E}(M)|^2 \ge |\widehat{E\times E}(0,\ldots, 0)|^2 = \frac{|E|^4}{q^{4d}}.$$
Then we see that 
\begin{equation}\label{WrKK} W(r)\ge \frac{|E|^4}{q} + \frac{|E|^4}{q^d} -q^{d-1}|E|^2 -\frac{|E|^4}{q^2} -q^{d-1}|E|^2-2 q^{\frac{d-3}{2}}|E|^3.\end{equation}
We now prove the statement of Theorem \ref{mainthm} (ii).
Suppose that $|E|\ge 3 q^{d/2}.$  Since the term $\frac{|E|^4}{q^d}$ in the RHS of \eqref{WrKK} is positive, we can ignore it when we find a lower bound of $W(r).$  More precisely, we have 
\begin{equation}\label{equFirdavs} W(r)\ge \frac{|E|^4}{q} -2q^{d-1}|E|^2 -\frac{|E|^4}{q^2} -2 q^{\frac{d-3}{2}}|E|^3.\end{equation}
Since $1=\frac{2}{45}+\frac{10}{45}+\frac{15}{45}+\frac{18}{45},$   we can write
$\frac{|E|^4}{q} = \frac{2}{45}\frac{|E|^4}{q}+\frac{10}{45}\frac{|E|^4}{q}+\frac{15}{45}\frac{|E|^4}{q}+\frac{18}{45}\frac{|E|^4}{q}.$  Therefore,   \eqref{equFirdavs} becomes 
$$ W(r)\ge \frac{2}{45}\frac{|E|^4}{q} + \left(\frac{10}{45}\frac{|E|^4}{q}-2q^{d-1}|E|^2\right)+\left(\frac{15}{45}\frac{|E|^4}{q} -\frac{|E|^4}{q^2}\right)+ \left(\frac{18}{45}\frac{|E|^4}{q} -2 q^{\frac{d-3}{2}}|E|^3\right).$$
Since $q\ge 3$ and $|E|\ge 3 q^{d/2}$,  each value in parentheses above is non-negative. Thus, the required estimate is obtained:
$$ W(r)\ge \frac{2}{45}\frac{|E|^4}{q}.$$


\subsection{Proof of Theorem \ref{mainthm} (iii)}
Suppose that $d$ is odd and  $r\in \mathbb F_q^*$ is not a square number, namely, $\eta(r)=-1.$ By \eqref{CaseC} of (Case C),  
$$W(r)\ge \frac{|E|^4}{q} - q^{3d} \sum_{M\in V_{\mathcal{Q}_r^*}} |\widehat{E\times E}(M)|^2 +q^{d-1}|E|^2 -\left(\frac{|E|^2}{q} +  q^{\frac{d-1}{2}} |E|  \right)^2.$$
Expand the last square term above and note that 
$$0\le \sum_{M\in V_{\mathcal{Q}_r^*}} |\widehat{E\times E}(M)|^2 \le \sum_{M\in \mathbb F_q^{2d}} |\widehat{E\times E}(M)|^2=\frac{|E|^2}{q^{2d}}. $$
When $r\in \mathbb F_q^*$ is non-square, we therefore get
\begin{equation}\label{Com1} W(r)\ge \frac{|E|^4}{q}-q^d |E|^2+q^{d-1}|E|^2 -\frac{|E|^4}{q^2} -q^{d-1}|E|^2-2 q^{\frac{d-3}{2}}|E|^3.\end{equation}
On the other hand, when $r\in \mathbb F_q^*$ is a square number, we already know from \eqref{WrKK} that 
\begin{equation}\label{Com2} W(r)\ge \frac{|E|^4}{q} + \frac{|E|^4}{q^d} -q^{d-1}|E|^2 -\frac{|E|^4}{q^2} -q^{d-1}|E|^2-2 q^{\frac{d-3}{2}}|E|^3.\end{equation}
Now let us compare the above two lower bounds of $W(r).$ 
It suffices to compare only the second and third terms of them since the rest of the terms are equal. Since $q\ge 3,$ it is not hard to see  that 
$$-q^d |E|^2+q^{d-1}|E|^2 < \frac{|E|^4}{q^d} -q^{d-1}|E|^2.$$
This implies that  the lower bound of $W(r)$ in \eqref{Com1} is much smaller than that in \eqref{Com2}.
Thus, the inequality \eqref{Com1} holds for all $r\in \mathbb F_q^*.$ In other words, for all $r\in \mathbb F_q^*,$
\begin{equation}\label{Allr} W(r)\ge \frac{|E|^4}{q}-q^d |E|^2 -\frac{|E|^4}{q^2} -2 q^{\frac{d-3}{2}}|E|^3.\end{equation}
Now we are ready to finish the proof of Theorem \ref{mainthm} (iii). Assume that $|E|\ge \frac{11}{6} q^{\frac{d+1}{2}}.$ 
Since $1=\frac{2}{363} + \frac{108}{363}+ \frac{121}{363}+\frac{132}{363},$ we can write 
$$ \frac{|E|^4}{q} = \frac{2}{363} \frac{|E|^4}{q} + \frac{108}{363}\frac{|E|^4}{q}+ 
\frac{121}{363}\frac{|E|^4}{q}+  \frac{132}{363}\frac{|E|^4}{q}.$$
Combining this with the above inequality \eqref{Allr}, we get
$$ W(r)\ge \frac{2}{363} \frac{|E|^4}{q} + \left( \frac{108}{363}\frac{|E|^4}{q}-q^d |E|^2\right) 
+ \left(\frac{121}{363}\frac{|E|^4}{q}-\frac{|E|^4}{q^2}\right)
+ \left( \frac{132}{363}\frac{|E|^4}{q}-2 q^{\frac{d-3}{2}}|E|^3 \right). $$
Since $q\ge 3$ and $|E|\ge \frac{11}{6} q^{(d+1)/2},$  we see from a direct computation  that
each value in parentheses above is non-negative. Hence, we obtain that
$W(r)\ge \frac{2}{363} \frac{|E|^4}{q},$ which proves Theorem \ref{mainthm} (iii).

\section{Proof of Propositions \ref{propFV} and \ref{lem2.3K}}
To efficiently prove the propositions, we begin by deducing a preliminary lemma, regarding the Fourier transform on the diagonal homogeneous variety of degree two in $\mathbb F_q^n.$
Denote $\mathbf{x}=(x_1, \ldots, x_n)$, $\mathbf{m}=(m_1, \ldots, m_n)\in \mathbb F_q^n.$
Let  $\mathbf{a}=(a_1,\ldots, a_n)\in \mathbb F_q^n$ such that 
$a_j\ne 0$ for all $j=1, 2, \ldots, n.$  Consider an algebraic variety 
$$ H_{\mathbf{a}}:=\left\{\mathbf{x}\in \mathbb F_q^n: \sum_{j=1}^n a_j x_j^2=0\right\}.$$
We call this variety $H_{\mathbf{a}}$ the diagonal homogeneous variety of degree two with a coefficient vector $\mathbf{a}$ in $\mathbb F_q^n.$

We also define the dual variety  of $H_{\mathbf{a}}$, denoted by $H_{\mathbf{a}^*}$, as 
$$H_{\mathbf{a}^*}:= \left\{\mathbf{m}\in \mathbb F_q^n: \sum_{j=1}^n a_j^{-1} m_j^2=0\right\}.$$

\begin{lemma}\label{FTHa} Let $H_{\mathbf{a}} $ be the diagonal homogeneous variety of degree two with a coefficient vector $\mathbf{a}$ in $\mathbb F_q^n.$ Then, for $\mathbf{m}\in \mathbb F_q^n,$  the Fourier transform on $H_{\mathbf{a}}$ is given as follows:
\begin{itemize}
\item [(i)] If $n$ is even, then 
$$ \widehat{H_{\mathbf{a}}}(\mathbf{m})
=q^{-1}\delta_0(\mathbf{m})+q^{-\frac{n}{2}} \eta\left( (-1)^{n/2} \prod_{j=1}^n a_j\right) H_{\mathbf{a}^*}(\mathbf{m})-q^{-(n+2)/2} \eta\left((-1)^{n/2}\prod_{j=1}^n a_j\right). $$

\item [(ii)] If $n$ is odd, then 
$$ \widehat{H_{\mathbf{a}}}(\mathbf{m})
=\left\{\begin{array}{ll} q^{-1} \delta_0(\mathbf{m}) \quad&\mbox{if}~~\mathbf{m}\in H_{\mathbf{a}^*},\\
q^{-\frac{n+1}{2}} \eta\left( (-1)^{(n+3)/2} \prod\limits_{j=1}^n a_j\right) \eta\left(\sum\limits_{j=1}^n a_j^{-1} m_j^2\right)\quad&\mbox{if}~~\mathbf{m}\notin H_{\mathbf{a}^*}.\end{array}\right.$$
\end{itemize}
\end{lemma}
\begin{proof} The proof uses the standard orthogonality argument of characters and Gauss sum estimates. It follows that 
\begin{align*}
\widehat{H_{\mathbf{a}}}(\mathbf{m})&= q^{-n} \sum_{x\in H_{\mathbf{a}} } \chi(\mathbf{m}\cdot \mathbf{x})\\
&=q^{-n} \sum_{\mathbf{x}\in {\mathbb F_q^n}}\left(q^{-1}\sum_{s\in {\mathbb F_q}} \chi\left( s(a_1x_1^2+\cdots+ a_nx_n^2)\right)\right)~\chi(\mathbf{m}\cdot \mathbf{x})\\
&= q^{-1} \delta_0(\mathbf{m})+  q^{-n-1}\sum_{\mathbf{x}\in {\mathbb F_q^n}}\sum_{s\neq 0} \chi\left( s(a_1x_1^2+\cdots+ a_n x_n^2)\right)~\chi(\mathbf{m}\cdot \mathbf{x})\\
&=q^{-1} \delta_0(\mathbf{m})+ q^{-n-1} \sum_{s\neq 0} \prod_{j=1}^n \sum_{x_j\in {\mathbb F_q}} \chi(sa_j x_j^2+m_jx_j).
\end{align*}
Summing over $x_j\in {\mathbb F_q}$ by the formula \eqref{CSQ}, we get
\begin{equation}\label{maineqKF} \widehat{H_{\mathbf{a}}}(\mathbf{m})=q^{-1}\delta_0(\mathbf{m})+ q^{-n-1} \mathcal{G}^n \eta(a_1\cdots a_n) \sum_{s\neq 0} \eta^n(s) \chi\left( -\frac{1}{4s}\left(\frac{m_1^2}{a_1}+\cdots+\frac{m_n^2}{a_n}\right)\right).\end{equation}

\begin{itemize}
\item [(i)] Suppose that $n$ is even. Since $\eta^{n}(s)=1,$ the sum over $s\ne 0$ is $(qH_{\mathbf{a}^*}(\mathbf{m})-1)$ by  application of the orthogonality of $\chi.$ 
By Lemma \ref{SGa}, note that $\mathcal{G}^n= (\eta(-1) q)^{n/2}=\eta( (-1)^{n/2}) q^{n/2}.$ So we obtain that
$$\widehat{H_{\mathbf{a}}}(\mathbf{m})=q^{-1}\delta_0(\mathbf{m})+ q^{-\frac{n+2}{2}} 
\eta((-1)^{n/2} a_1\cdots a_n) \left(qH_{\mathbf{a}^*}(\mathbf{m})-1\right).$$

Hence, the statement of Theorem (i) is proven.
\item [(ii)] Suppose that $n$ is odd. Then $ \eta^n=\eta$, and 
so, if 
$\frac{m_1^2}{a_1}+\cdots+\frac{m_n^2}{a_n}=0$, namely $\mathbf{m}\in H_{\mathbf{a}^*},$ then $ \widehat{H_{\mathbf{a}}}(\mathbf{m})=q^{-1}\delta_0(\mathbf{m})$ by the orthogonality of $\eta$.  On the other hand, when $\mathbf{m}\notin H_{\mathbf{a}^*},$ note by \eqref{Gauss} that the sum over $s\ne 0$ in \eqref{maineqKF} is equal to 
$\eta\left( -\frac{1}{4}\left(\frac{m_1^2}{a_1}+\cdots+\frac{m_n^2}{a_n}\right)\right) \mathcal{G}.$
Also note that $\eta(-4^{-1})=\eta(-1).$  Then we get
$$ \widehat{H_{\mathbf{a}}}(\mathbf{m})=q^{-n-1} \mathcal{G}^{n+1}\eta(-1) \eta(a_1\cdots a_n)\eta\left(\frac{m_1^2}{a_1}+\cdots+\frac{m_n^2}{a_n}\right) .$$
Since $\mathcal{G}^2=\eta(-1)q$ by Lemma \ref{SGa},  we see that
$$ \mathcal{G}^{n+1} \eta(-1)= (\eta(-1)q)^{(n+1)/2} \eta(-1)= \eta\left( (-1)^{(n+3)/2}\right) q^{(n+1)/2}.$$
Inserting this into the above equation, we obtain the required value of $\widehat{H_{\mathbf{a}}}(\mathbf{m}).$ 
\end{itemize}
\end{proof}

Now we are ready to prove Proposition \ref{propFV}, which will be a special case of Lemma \ref{FTHa} (i).
\subsection{Proof of Proposition \ref{propFV}}  
Let $x, x'\in \mathbb F_q^d.$ Then we see that 
$$ V_{\mathcal{Q}_r}=\{(x,x')\in \mathbb F_q^{2d}: ||x||_Q-r||x'||_Q=0\}.$$
\begin{itemize}
\item [(i)] Suppose that $d$ is even. Then using Definition \ref{defnormQ} (i), the coefficient vector of  $||x||_Q$ is 
$\mathbf{a}'$
 $:=(1, -1, \ldots, 1, -\varepsilon)\in \mathbb F_q^d$ and  that of $-r||x'||_Q$ is  $\mathbf{a}'':=(-r, r, \ldots, -r, r\varepsilon)\in \mathbb F_q^d.$ Hence, 
$V_{\mathcal{Q}_r}$ is the diagonal homogeneous variety of degree two with a coefficient vector $\mathbf{a}=(\mathbf{a}', \mathbf{a}'')$ in $\mathbb F_q^{2d}.$ Thus, with $n=2d$ even, we can apply the part (i) of Lemma \ref{FTHa}, where $\mathbf{m}=M,$  
$H_{\mathbf{a}}=V_{\mathcal{Q}_r},$  $H_{\mathbf{a}^*}=V_{\mathcal{Q}_r^*},$ and 
$$ \eta\left(\prod_{j=1}^n a_j\right)= \eta\left((-1)^d r^d \varepsilon^2\right)=1.$$
Consequently we obtain that
$$ \widehat{V_{\mathcal{Q}_r}}(M)
=q^{-1}\delta_0(M)+q^{-d} \eta((-1)^{d})\eta\left((-1)^d r^d \varepsilon^2\right) V_{\mathcal{Q}_r^*}(M)-q^{-(d+1)} \eta((-1)^{d})\eta\left((-1)^d r^d \varepsilon^2\right). $$
Since $d$ is even,  each value of the above quadratic characters $\eta$ is 1 and we complete the proof of Proposition  \ref{propFV} (i).
\item [(ii)] Suppose that $d$ is odd. By Definition \ref{defnormQ} (ii), we see that $\mathbf{a}=(\mathbf{a}', \mathbf{a}'')\in \mathbb F_q^{2d}$ such that
$\mathbf{a}'=(1, -1, \ldots, 1, -1, \varepsilon)\in \mathbb F_q^d,$ and $\mathbf{a}''=(-r, r, \ldots, -r, r, -r\varepsilon) \in \mathbb F_q^d.$ Hence, the proof is the same as the case of (i) except that for odd $d,$
$$ \eta\left(\prod_{j=1}^n a_j\right)= \eta\left((-1)^{d} r^d \varepsilon^2\right)=\eta(-r).$$
Thus, the conclusion of Proposition \ref{propFV} (ii) follows.
\end{itemize}
\subsection{Proof of Proposition \ref{lem2.3K}}
Since $w(0):=\sum\limits_{\substack{x,y\in E:\\ ||x-y||_Q=0}}1$,  it follows from the general counting lemma (Lemma \ref{GCL}) that 
\begin{equation}\label{CFw0} w(0)=q^{2d} \sum_{m\in \mathbb F_q^d} \widehat{(S_Q)_0}(m) |\widehat{E}(m)|^2,\end{equation}
where $(S_{Q})_0=\{x\in \mathbb F_q^d: ||x||_Q=0\}.$ 
We will invoke the following explicit estimate on the Fourier transform on $(S_{Q})_0,$ which can be proven by Lemma \ref{FTHa}.
\begin{corollary} \label{cor6.2F}
Let $m\in \mathbb F_q^d.$ 
\begin{itemize} 
\item [(i)] If $d$ is even, then
$$ \widehat{(S_{Q})_0}(m)=q^{-1}\delta_0(m) + q^{-d/2}\eta(\varepsilon)  (S_{Q^*})_0(m)- q^{-(d+2)/2} \eta(\varepsilon).$$
\item[(ii)] If $d$ is odd, then
$$ \widehat{(S_{Q})_0}(m)
=\left\{\begin{array}{ll} q^{-1} \delta_0(m) \quad&\mbox{if}~~m\in (S_{Q^*})_0,\\
q^{-\frac{d+1}{2}} \eta(\varepsilon) \eta\left(\sum\limits_{j=1}^{d-1} (-1)^{j+1} m_j^2 + \varepsilon^{-1} m_d^2\right)\quad&\mbox{if}~~m\notin (S_{Q^*})_0.\end{array}\right.$$
\end{itemize} 
\end{corollary}
\begin{proof} Let $\mathbf{a}=(a_1, \ldots, a_d)\in \mathbb F_q^d$ be the coefficient vector of $||x||_Q$, $x\in \mathbb F_q^d.$ Then it is clear that $(S_{Q})_0$ is the diagonal homogeneous variety of degree two with a coefficient vector $\mathbf{a}$ in $\mathbb F_q^d.$
\begin{itemize}
\item [(i)] Assume that $d$ is even. Then, by definition of $||x||_Q$, we see that 
$\mathbf{a} =(1,-1, \ldots, 1, -\varepsilon)$ in $\mathbb F_q^d.$  Since $\prod\limits_{j=1}^d a_j=(-1)^{d/2}\varepsilon,$  the first part of the corollary follows by applying Lemma \ref{FTHa} (i). 

\item[(ii)] Assume that $d$ is odd. Then $\mathbf{a} =(1,-1, \ldots, 1, -1, \varepsilon)$ in $\mathbb F_q^d$ and so
$$ \prod_{j=1}^da_j= (-1)^{(d-1)/2}\varepsilon.$$
Now applying Lemma \ref{FTHa} (ii), we obtain the statement of the second part of the corollary.
\end{itemize}
\end{proof}

We are now ready to complete the proof of Proposition \ref{lem2.3K}.  By the definition of $w(0)$, it is clear that $w(0)$ is a non-negative integer, namely, $w(0)\ge 0.$ So it remains to find the required upper bounds for $w(0).$ 

\subsubsection{Proof of Proposition \ref{lem2.3K}-(i), (ii)}
Suppose that $d$ is even. Then inserting the value of $\widehat{(S_Q)_0}(m)$ in Corollary \ref{cor6.2F} (i) into the formula \eqref{CFw0},   we see from a direct computation that 
$$w(0)=\frac{|E|^2}{q} + q^{3d/2} \eta\left(\varepsilon\right) \sum_{m\in (S_{Q^*})_0} |\widehat{E}(m)|^2-q^{(d-2)/2} |E| \eta\left(\varepsilon\right),$$
where we also used the Plancherel theorem \eqref{basicP} with the simple fact that $\widehat{E}(0,\ldots, 0)=q^{-d}|E|.$
Notice that the sign of the second term above is different from that of the third term since $\eta\left(\varepsilon\right)=\pm 1$ and $\sum\limits_{m\in (S_{Q^*})_0} |\widehat{E}(m)|^2\ge 0.$ So, to estimate an upper bound of $w(0),$  we can ignore the negative term which is exactly  one of them. By this way, we easily obtain the required upper bounds in Proposition \ref{lem2.3K} (i), (ii).

\subsubsection{Proof of Proposition \ref{lem2.3K}-(iii)} Suppose that $d$ is odd. It is clear from \eqref{CFw0} that
$$ 0\le w(0) \le q^{2d} \sum_{m\in \mathbb F_q^d} \left|\widehat{(S_Q)_0}(m)\right| |\widehat{E}(m)|^2.$$
On the other hand,  Corollary \ref{cor6.2F} (ii) implies that
$$ \left|\widehat{(S_Q)_0}(m)\right|=\left\{ \begin{array}{ll} q^{-1} \delta_0(m) &\quad\mbox{if}~~m\in  (S_{Q^*})_0,\\
 q^{-\frac{d+1}{2}} &\quad\mbox{if}~~ m\notin (S_{Q^*})_0.\end{array} \right.$$
Thus, combining the above facts, we obtain the desired estimate as follows:
\begin{align*} w(0)&\le q^{2d} \sum_{m\in \mathbb F_q^d} q^{-1} \delta_0(m) |\widehat{E}(m)|^2 + q^{2d} \sum_{m\in \mathbb F_q^d} q^{-\frac{d+1}{2}} |\widehat{E}(m)|^2\\
&= q^{2d-1} |\widehat{E}(0,\ldots, 0)|^2 + q^{2d}q^{-\frac{d+1}{2}}\sum_{m\in \mathbb F_q^d} |\widehat{E}(m)|^2\\
&= \frac{|E|^2}{q} + q^{\frac{d-1}{2}}|E|.\end{align*}


\end{document}